\documentclass[10pt]{article}
\usepackage{arxiv}

\usepackage{microtype}
\usepackage{graphicx}
\usepackage{subfigure}
\usepackage{booktabs} 
\usepackage{hyperref}
\usepackage{algorithm,algorithmic}

\usepackage{amsmath,amssymb,amsthm}
\usepackage{pifont}

\newtheorem{thm}{Theorem}[section]

\newtheorem{lem}[thm]{Lemma}
\newtheorem{assum}[thm]{Assumption}
\newtheorem{defn}[thm]{Definition}
\newtheorem{rem}[thm]{Remark}
\newtheorem{prop}[thm]{Proposition}

\newcommand{\Sum}{\displaystyle\sum\limits}

\newcommand{\ZZ}{\mathbb{Z}}
\newcommand{\RR}{\mathbb{R}}

\newcommand{\eps}{\varepsilon}

\newcommand{\ol}{\overline}

\newcommand{\circledOne}{\text{\ding{172}}}
\newcommand{\circledTwo}{\text{\ding{173}}}

\newcommand{\norm}[1]{\left \| #1 \right\|}

\newcommand{\onevec}{\mathbf{1}}

\renewcommand{\le}{\leqslant}
\renewcommand{\ge}{\geqslant}
\renewcommand{\hat}{\widehat}

\newcommand\numberthis{\addtocounter{equation}{1}\tag{\theequation}}

\DeclareMathOperator*{\argmin}{arg\,min}

\DeclareMathOperator{\proj}{Proj}

\graphicspath{{./figures/}}

\author{
	Alexander Rogozin \\
	Moscow Institute of Physics and Technology \\
	\texttt{aleksandr.rogozin@phystech.edu} \\
	 \And
	Alexander Gasnikov \\
	Moscow Institute of Physics and Technology \\
	\texttt{gasnikov@yandex.ru}
}

\title{Projected Gradient Method for Decentralized Optimization over Time-Varying Networks}
\date{}

\begin{document}

\maketitle

\begin{abstract}
Decentralized distributed optimization over time-varying networks is nowadays a very popular branch of research in optimization theory and consensus theory. A motivation to consider such networks is an application to drone or satellite networks, as well as large scale machine learning. Communication complexity of a decentralized optimization algorithm depends on condition numbers of objective function and communication matrix. In this paper, we provide an algorithm based on projected gradient descent which improves existing geometric rates in the literature in terms of function condition number. This result is obtained due to robustness of non-accelerated optimization schemes to changes in network. Moreover, we study the performance of accelerated methods in a time-varying setting.
\end{abstract}

\section{Introduction}

The theory of decentralized distributed optimization goes back to \cite{bertsekas1989parallel}. In the last few years this branch of research has aroused great interest in optimization community. A set of papers proposing optimal algorithms for convex optimization problems of sum-type has appeared. See for example \cite{arjevani2015communication,bach2017optimal,lan2017communication,dvinskikh2019decentralized} and references therein. In all these papers, authors consider sum-type convex target functions and aim at proposing algorithms that find the solution with required accuracy and make the best possible number of communication steps and number of oracle calls (gradient calculations of terms in the sum). In \cite{dvinskikh2019decentralized}, it is mentioned that the theory of optimal decentralized distributed algorithms looks very close to the analogous theory for ordinary convex optimization (\cite{Nemirovskii1983,nesterov2013introductory,bub15}). Roughly speaking, in a first approximation, decentralized distributed optimization comes down to the theory of optimal methods and this theory is significantly based on the theory of non-distributed optimal methods. 

In decentralized distributed optimization over time-varying graphs, another situation takes place. The communication network topology changes from time to time, which can be caused by technical malfunctions such as loss of connection between the agents. Due to the many applications, the interest to these class of problems has grown significantly during the last few years. There appears a number of papers with theoretical analysis of rate of convergence for convex type problems: \cite{nedic2014distributed,nedic2017geometrically,Maros2018,lu2018geometrical,vandistributed,shi15extra,Qu2017accelerated,shi2014on_the_linear}. But there is still a big gap between the theory for decentralized optimization on fixed graphs and the theory over time-varying graphs. The attempt to close this gap (specifically, to develop optimal methods) for the moment required very restricted additional conditions(\cite{rogozin2019optimal}). 

In this paper, we make a step in the direction of development of optimal methods over time-varying graphs: we propose non-accelerated gradient descent for smooth strongly convex target functions of sum-type. Our analysis is based on running a projected gradient method over the constraint set, which is a hyperplane. In order to solve auxiliary problem (to find a projection on hyperplane) we use non accelerated consensus type algorithms (see \cite{hendrikx2018accelerated} and references therein for comparison). Note that proposed analysis of external non-accelerated gradient descent method can be generalized for the case of accelerated gradient method. We plan to do it in subsequent works.

This paper is organized as follows. In Section \ref{section:approximating_projected_gd} we briefly explain how our method approximates projected gradient descent. In Section \ref{section:preliminaries}, we recall some basic definitions and introduce assumptions for a mixing matrix. After that, we introduce decentralized projected gradient method in Section \ref{section:decentralized_proj_gd} provide its convergence results. Finally, numerical experiments and comparison to other methods can be found in Section \ref{section:experiments}.

\section{Approximating Projected Gradient}\label{section:approximating_projected_gd}

Consider convex minimization problem of sum-type:
\begin{align}\label{problem:initial}
f(x) = \Sum_{i=1}^n f_i(x) \longrightarrow \min_{x\in\RR^d}.
\end{align}
The summands $f_i$ need only be differentiable and may not be convex. We seek to solve problem \eqref{problem:initial} in a decentralized setup, so that every node locally holds $f_i$ and may exchange data with its neighbors. Moreover, we are interested in the time-varying case. This means the communication network changes with time and is represented by a sequence of graphs $\{\mathcal{G}_k\}_{k=1}^{\infty}$. 

Let us reformulate problem \eqref{problem:initial} in a following way.
\begin{align}\label{problem:F_constrained}
F(X) = \Sum_{i=1}^n f_i(x_i) \longrightarrow \min_{x_1 = \ldots = x_n}.
\end{align}
Here $X\in\RR^{d\times n}$ is a matrix consisting of columns $x_1, \ldots, x_n$. The above representation means that local copies $x_i$ of parameter vector $x$ are distributed over the agents in the network. Now, if every node computes $\nabla f_i(x_i)$, then the gradient $\nabla F(X) = [\nabla f_1(x_1),\ldots, \nabla f_n(x_n)]$ will be distributed all over the network. We will use notation $\nabla F(X)$ in the analysis, although $\nabla F(X)$ is not stored at one computational entity.

We call $K$ a linear subspace in $\RR^{d\times n}$ determined by the constraint $x_1 = \ldots = x_n$. Consider a projected gradient method applied to problem \eqref{problem:F_constrained}.
\begin{align}\label{alg:exact_projected_gd}
\Pi_{k+1} = \Pi_k - \gamma\proj_K(\nabla F(\Pi_k))
\end{align}
Since $K$ is a linear subspace, projection operator $\proj_K(\cdot)$ is linear. Therefore, update rule \eqref{alg:exact_projected_gd} is equivalent to 
\begin{align*}
	\Pi_{k+1} = \proj_K(\Pi_k - \gamma\nabla F(\Pi_k)).
\end{align*}

Choosing $\Pi_0\in K$ makes the method trajectory stay in $K$, and the algorithm may be interpreted as a simple gradient descent on $K$. However, exact projected method cannot be run in a decentralized manner. Instead, projections can be computed via consensus algorithm, which we describe further in the paper. This leads to inexact projected gradient descent, i.e. the algorithm of type
\begin{align}\label{eq:inexact_projected_gd}
	\Pi_{k+1} \approx \Pi_k - \gamma\proj_K(\nabla F(\Pi_k)),
\end{align}
where approximate computation of $\proj_K(\nabla F(\Pi_k))$ is performed with pre-defined accuracy $\eps_1$ which depends on desired accuracy $\eps$. A proper choice of $\eps_1$ makes procedure \eqref{eq:inexact_projected_gd} approximate projected gradient method \eqref{alg:exact_projected_gd} well.

\section{Preliminaries}\label{section:preliminaries}

\subsection{Strong Convexity and smoothness}
Strongly convex and smooth functions are the focus of this paper.
\begin{defn}
	Let $\mathbb{X}$ be either $\RR^d$ with $2$-norm or $\RR^{d\times n}$ with Frobenius norm. A differentiable function $f:\mathbb{X}\rightarrow\RR$ is called
	\begin{itemize}
		\item \textbf{convex}, if for any $x, y\in\mathbb{X}$
		\begin{align*}
			f(y)\ge f(x) + \langle\nabla f(x), y - x\rangle;
		\end{align*}
		\item \textbf{$\mu$-strongly convex}, if for any $x, y\in\mathbb{X}$
		\begin{align*}
			f(y)\ge f(x) + \langle\nabla f(x), y - x\rangle + \frac{\mu}{2}\|y - x\|^2;
		\end{align*}
		\item \textbf{$L$-smooth}, if its gradient $\nabla f(x)$ is $L$-Lipschitz, i.e. for any $x, y\in\mathbb{X}$
		\begin{align*}
			\|\nabla f(y) - \nabla f(x)\|\le L\|y - x\|,
		\end{align*}
		or, equivalently, for all $x, y\in\mathbb{X}$
		\begin{align*}
			f(y)\le f(x) + \langle\nabla f(x), y - x\rangle + \frac{L}{2}\|y - x\|^2.
		\end{align*}
	\end{itemize}
\end{defn}

\begin{prop}\label{prop:nesterov_strongly_convex_smooth}
	For $\mu$-strongly convex $L$-smooth function  $f$, it holds
	\begin{align*}
	&\langle \nabla f(x) - \nabla f(y), x - y\rangle \\
	&\qquad \ge \frac{\mu L}{\mu + L}\|x - y\|^2 + \frac{1}{\mu + L} \|\nabla f(x) - \nabla f(y)\|^2
	\end{align*}
\end{prop}
\begin{proof}
	See Theorem 2.1.11 in \cite{nesterov2013introductory}.
\end{proof}

\subsection{Mixing Matrix}

In this paper, a communication network is given by a sequence of graphs $\{\mathcal{G}_k\}_{k=0}^{\infty} = (V_k, E_k)$. The network restrictions are induced by edge sets at every step and represented by sequence $\{W(k)\}_{k=0}^{\infty}$ of mixing matrices. It is not necessary that all graphs $\{\mathcal{G}_k\}_{k=0}^{\infty}$ are connected. Instead, we are interested in a sub-sequences of finite length and denote
\begin{align*}
	W_b(k) = W(k) W(k - 1)\ldots W(k - b + 1)
\end{align*}
for $b\in\ZZ_{++}, k\in\ZZ_+$. We also define $W_0(k) = I$ for any $k$. The following assumption is typical for analysis of consensus algorithms (for example, see Assumption 3.1 in \cite{Nedic2017achieving})

\begin{assum}\label{assum:mixing_matrix}
	For every $k\in\ZZ_{+}$, mixing matrix $W(k)$ has the following properties:
	\begin{enumerate}
		\item If $i\ne j$ and $(i, j)\notin E_k$, then $W_{ij}(k) = 0$.
		\item $W(k)$ is doubly stochastic: $W(k) \mathbf{1} = \mathbf{1},~ \mathbf{1}^\top W(k) = \mathbf{1}^\top$.
		\item There exists some $B\in\ZZ_{++}$, such that 
		\begin{align*}
			\delta := \sup_{k\ge B - 1} \delta(k) < 1, 
		\end{align*}
		where 
		\begin{align*}
			\delta(k) = \sigma_{\max}\left[W_B(k) - \frac{1}{n} \onevec\onevec^\top\right]
		\end{align*}
	\end{enumerate}
\end{assum}

\section{Decentralized Projected Gradient Method}\label{section:decentralized_proj_gd}

\subsection{Finding Inexact Projection}

Finding inexact projection on $K$ is equivalent to computing average of vectors held by agents over the network. In other words, for $X = [x_1,\ldots,x_n]$ it holds
\begin{align*}
	\proj_K(X) = [\ol x,\ldots,\ol x],~ \ol x = \frac{1}{n}\sum_{i=1}^n x_i
\end{align*}

This is done by the following algorithm.

\begin{algorithm}[ht]
	\caption{Consensus}
	\begin{algorithmic}[1]
		\REQUIRE Each node holds $x_i^0$ and iteration number $N$.
		\FOR{$k=0,1,2,\cdots,N-1$}
		\STATE{$X_{k+1} = X_k W_k$}
		\ENDFOR
	\end{algorithmic}
	\label{alg:consensus}
\end{algorithm}
Algorithm \ref{alg:consensus} is robust to changes in network topology \cite{Nedic2017achieving}. Its convergence rate is given by the following Proposition (i.e. see \cite{Nedic2017achieving} Lemma 3.4).
\begin{prop}\label{prop:consensus}
	Let Assumption \ref{assum:mixing_matrix} hold and let $a = W_B(k) b$. Then
	\begin{align*}
	\norm{a - \proj_K(a)}\le \delta \norm{b - \proj_K(b)}.
	\end{align*}
\end{prop}

\subsection{Problem Reformulation and Assumptions}
Recall problem \eqref{problem:initial}:
\begin{align*}
	f(x) = \Sum_{i=1}^n f_i(x) \longrightarrow \min_{x\in\RR^d}
\end{align*}
and its reformulation \eqref{problem:F_constrained}:
\begin{align*}
	F(X) = \Sum_{i=1}^n f_i(x_i) \longrightarrow \min_{x_1 = \ldots = x_n}.
\end{align*}

\begin{assum}\label{assum:function}
	\item
	\begin{enumerate}
		\item Function $f$ is $\mu_f$-strongly convex and $L_f$-smooth.
		\item Every function $f_i$ is differentiable.
	\end{enumerate}
\end{assum}

Update rule \eqref{eq:inexact_projected_gd} makes iterates stay close to subspace $K$ and therefore approximate projected gradient method. Convergence rate of gradient descent depends on objective function condition number. Below we illustrate that condition number of $F$ on $\RR^{d\times n}$ is worse than condition number on $K$. To do this, we assume that every $f_i$ is $\mu_i$-strongly convex and $L_i$-smooth. Note that it is done solely for illustration and is not required in Assumption \ref{assum:function}. Consider $X, Y\in K$: $X = (x, \ldots, x), Y = (y, \ldots, y)$.
\begin{align*}
	&F(X) = \Sum_{i=1}^n f_i(x) = f(x),\ F(Y) = f(y), \\
	&\|Y - X\|^2 = n\|y - x\|^2 \\
	&\langle \nabla F(X), Y - X\rangle = \Sum_{i=1}^n \langle \nabla f_i(x), y - x\rangle = \langle \nabla f(x), y - x\rangle \\
	&F(Y)\ge F(X) + \langle \nabla F(X), Y - X\rangle + \frac{\mu_f}{2n}\|Y - X\|^2 \\
	&F(Y)\le F(X) + \langle \nabla F(X), Y - X\rangle + \frac{L_f}{2n}\|Y - X\|^2.
\end{align*}
It follows that $F$ is $L_f/(2n)$-smooth and $\mu_f/(2n)$-strongly convex on $K$. On the other hand, $F$ is $L_{\max}$-smooth and $\mu_{\min}$-strongly convex on $\RR^{d\times n}$, where
\begin{align}\label{def:mu_L_minmax_and_average}
&\mu_{\min} = \min_i\mu_i,\ L_{\max} = \max_i L_i.
\end{align}
Condition number of $F$ on $K$ is $L_f / \mu_f$, which is better than $L_{\max} / \mu_{\min}$. This is beneficial for the resulting convergence rate.

\subsection{Inexact Projected Gradient Descent}

\begin{algorithm}[H]
	\caption{Decentralized Projected GD}
	\begin{algorithmic}[1]
		\REQUIRE Each node holds $f_i(\cdot)$ and iteration number $N$.
		\STATE{Initialize $X_0 = [x_0, \ldots, x_0]$, choose $c > 0$.}
		\FOR{$k=0,1,2,\cdots,N-1$}
		\STATE{$Y_{k+1} = X_k - \gamma\nabla F(X_k)$}
		\STATE{$X_{k+1}\approx \proj_K (Y_{k+1})$ with accuracy $\eps_1$, i.e. $\|X_{k+1} - \proj_K (Y_{k+1})\|^2\le \eps_1$ and $X_{k+1} - \proj_K(Y_{k+1})\in K^{\bot}$}\label{state:compute_proj}
		\ENDFOR
	\end{algorithmic}
	\label{alg:decentralized_projected_GD}
\end{algorithm}

Performing projection step in a decentralized way on a time-varying graph is done by Algorithm \ref{alg:consensus}. Here we present a convergence result for Algorithm \ref{alg:decentralized_projected_GD}.

\begin{thm}\label{th:decentralized_gd}
	After $N = O\left(\frac{L_f}{\mu_f} \log\left(\frac{r_0^2}{\eps}\right)\right)$ iterations, Algorithm \ref{alg:decentralized_projected_GD} with $\eps_1 = \frac{\mu_f^2}{13n^2L_{\max}^2}\eps$ yields $X_N$ such that
	\begin{align*}
		\|X_N - X^*\|^2 \le \eps
	\end{align*}
\end{thm}
The proof of Theorem \ref{th:decentralized_gd} is performed in Appendix \ref{ap:decentralized_gd}.

\subsection{Overall Complexity}

Summarizing the results of Theorem \ref{th:decentralized_gd} and Proposition \ref{prop:consensus}, we get the final iteration complexity result.
\begin{thm}\label{th:summarize_rate}
	Algorithm \ref{alg:decentralized_projected_GD} with $\eps_1 = \frac{\mu_f^2}{13n^2 L_{\max}^2}\eps$ requires $N = O\left(\frac{{L_f}}{{\mu_f}} B \left(\log(\frac{1}{\delta})\right)^{-1} \log\left(\frac{\left(\|\nabla F(X^*)\| + L_{\max}\norm{X_0 - X^*}\right) n^2 L_{\max}^2}{\eps\mu_f^2}\right) \log\left(\frac{r_0^2}{\eps}\right)\right)$ communication steps, including sub-problem solution, to yield $X_N$ such that
	\begin{align*}
		\|X_N - X^*\|^2 \le \eps.
	\end{align*}
\end{thm}
\begin{rem}
	The convergence rate depends on $L_f$ and $\mu_f$ instead of $\mu_{\text{sum}} = \sum_{i=1}^n \mu_i$ and $L_{\text{sum}} = \sum_{i=1}^n L_i$. First, note that $\mu_f\ge \ol\mu$ and $L_f\le \ol L$. Second, and most importantly, the ratio $L_{\text{sum}} / L_f$ may be of magnitude $n$, and the ratio $\mu_f / \mu_{\text{sum}}$ may be arbitrary large. We illustrate this observation with the following example.
	\begin{align*}
		f(x) &= \frac{1}{2} (1+\alpha)\|x\|^2, \alpha > 0; \\
		f_i(x) &= \frac{1}{2} x_i^2 + \frac{\alpha}{2n}\|x\|^2.
	\end{align*}
	In this particular case, each $f_i(x)$ has $\mu_i = \alpha/n$ and $L_i = 1 + \alpha/n$, and therefore $L_{\text{sum}} = n + \alpha, \mu_{\text{sum}} = \alpha$. On the other hand, $\mu_f = L_f = 1 + \alpha$. Hence, 
	\begin{align*}
		\frac{L_{\text{sum}}}{L_f} &= \frac{n + \alpha}{1 + \alpha} \overset{\alpha\rightarrow +0}{\longrightarrow} n, \\
		\frac{\mu_f}{\mu_{\text{sum}}} &= \frac{1 + \alpha}{\alpha} \overset{\alpha\rightarrow +0}{\longrightarrow} \infty.
	\end{align*}
	The bound obtained in Theorem \ref{th:summarize_rate} is based on $L_f / \mu_f$. The example above shows that using this ratio in the bound may be significantly better than using $L_{\text{sum}} / \mu_{\text{sum}}$.
\end{rem}

\subsection{Extension to Accelerated Gradient Descent}

Algorithm \ref{alg:decentralized_projected_GD} includes gradient descent in the outer loop. It is possible to employ an accelerated scheme instead of a non-accelerated method, which leads to the following algorithm.
\begin{algorithm}\label{alg:decentralized_projected_acc_GD}
	\caption{Decentralized Accelerated Projected GD}
	\begin{algorithmic}[1]
		\REQUIRE{Each node holds $f_i(\cdot)$ and iteration number $N$.}
		\FOR{$k = 0, 1, \ldots, N - 1$}
		\STATE{$Y_{k+1} = X_k - \frac{1}{L}\nabla F(X_k)$}
		\STATE{$\tilde Y_{k+1}\approx \proj_K(Y_{k+1})$ with accuracy $\eps_1$}
		\STATE{$X_{k+1} = \tilde Y_{k+1} + \frac{\sqrt\kappa - 1}{\sqrt\kappa + 1}(\tilde Y_{k+1} - \tilde Y_k)$}
		\ENDFOR
	\end{algorithmic}
\end{algorithm}
Here $L$ and $\kappa$ denote the smoothness constant and condition number of $F$, respectively. The theoretical analysis of this method, including the choice of $\eps_1$, is left for future work. However, our numerical tests in Section \ref{section:experiments} show that Algorithm \ref{alg:decentralized_projected_acc_GD} outperforms both Algorithm \ref{alg:decentralized_projected_GD} and DIGing \cite{Nedic2017achieving}.

\section{Numerical Experiments}\label{section:experiments}

In this section, we provide numerical simulations of Algorithm \ref{alg:decentralized_projected_GD} on \textit{logistic regression} problem on LibSVM datasets (\cite{Chang2011}). The objective function is defined as
\begin{align*}
	f(x) = \frac{1}{m} \Sum_{i=1}^m \log\left[1 + \exp(-c_i(\langle a_i, x) + b_i))\right],
\end{align*}
where $a_i\in\RR^d$ are training samples and $c_i\in\{0, 1\}$ are class labels. In decentralized scenario, the training dataset is distributed between the agents in the network.

One of the tuned parameters of Algorithm \ref{alg:decentralized_projected_GD} is the number of inner iterations. On Figures \ref{fig:a9a} and \ref{fig:w8a}, we illustrate different choices of this parameter, and Proj-GD-$k$ denotes projected gradient method with $k$ iterations on each sub-problem. Moreover, we compare our algorithm to DIGing (\cite{Nedic2017achieving}).

\begin{figure}[ht]
	\centering
	\includegraphics[width=0.7\textwidth]{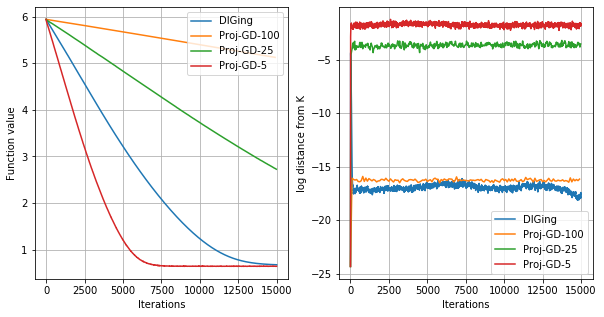}
	\caption{Random graph with $100$ nodes, \textsc{a9a} dataset. \label{fig:a9a}
	}
\end{figure}

\begin{figure}[ht]
	\centering
	\includegraphics[width=0.7\textwidth]{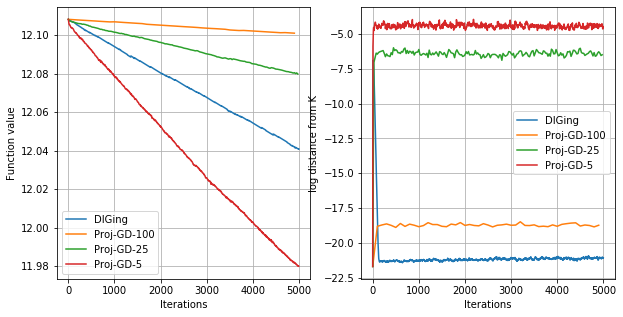}
	\caption{Random graph with $100$ nodes, \textsc{w8a} dataset.}
	\label{fig:w8a}
\end{figure}

Figures \ref{fig:a9a} and \ref{fig:w8a} suggest that performance of Algorithm \ref{alg:decentralized_projected_GD} is significantly dependent on the number of iterations made on projection step. A large number of steps results in more precise projection procedure, but also requires takes more communication steps. In other words, there is a trade-off between the number of communications and projection accuracy. In practice, one can tune number of iterations for sub-problem and find an optimal value for a specific practical case.

Moreover, we experiment with accelerated gradient descent and compare it to DIGing method. We find that accelerated method performs significantly better (Figure \ref{fig:ijcnn1}).
\begin{figure}[ht]
	\centering
	\includegraphics[width=0.7\textwidth]{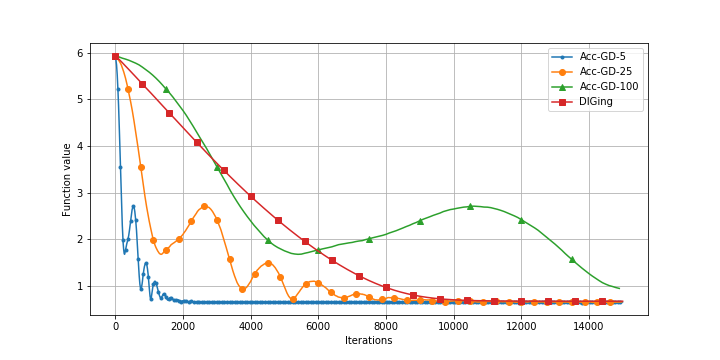}
	\caption{Random graph with $100$ nodes, \textsc{ijcnn1} dataset.}
	\label{fig:ijcnn1}
\end{figure}

We also run numerical comparisons with EXTRA \cite{shi15extra} and Acc-DNGD-SC \cite{Qu2017accelerated}, which are primal methods for decentralized optimization. To the best of our knowledge, these methods do not have theoretical guarantees in the time-varying case, unlike Algorithm \ref{alg:decentralized_projected_GD}. However, they still work on a time-varying network. Acc-DNGD-SC is outperformed by Proj-GD-5, while EXTRA is comparable with Algorithm \ref{alg:decentralized_projected_acc_GD}.
\begin{figure}[ht]
	\centering
	\includegraphics[width=0.7\textwidth]{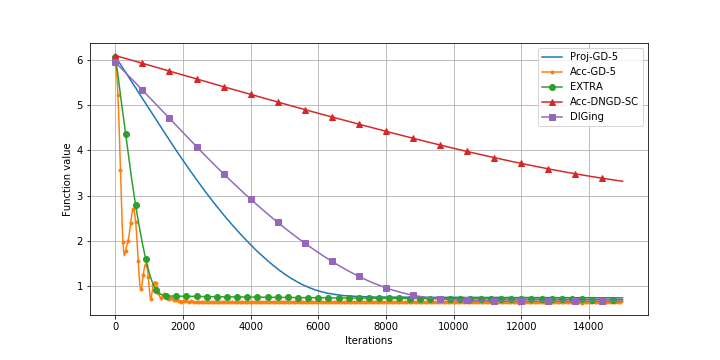}
	\caption{Random graph with 100 nodes, \textsc{ijcnn1} dataset.}
	\label{fig:methods_compared}
\end{figure}

\section{Conclusions and Future Work}

Our main result is based on a simple idea -- running projected gradient method with inexact projections. This idea is applied to decentralized optimization on time-varying graphs. The proposed method incorporates two different algorithms: projected gradient descent and obtaining mean of values held by agents over the network. The whole procedure is shown to be robust to network changes since non-accelerated schemes are used both for outer and inner loops. 

However, the question whether it is possible to employ an accelerated method either for finding projection or for running the outer loop remains open. Moreover, projection may be performed by a variety of algorithms, including randomized and asynchronous gossip algorithms (\cite{boyd2006randomized}). Investigation of new techniques for finding projection is left for future work.

\bibliography{references}

\begin{thebibliography}{10}

\bibitem{arjevani2015communication}
Y.~Arjevani and O.~Shamir.
\newblock Communication complexity of distributed convex learning and
  optimization.
\newblock In {\em Advances in neural information processing systems}, pages
  1756--1764, 2015.

\bibitem{bertsekas1989parallel}
D.~P. Bertsekas and J.~N. Tsitsiklis.
\newblock {\em Parallel and distributed computation: numerical methods},
  volume~23.
\newblock Prentice hall Englewood Cliffs, NJ, 1989.

\bibitem{boyd2006randomized}
S.~Boyd, A.~Ghosh, B.~Prabhakar, and D.~Shah.
\newblock Randomized gossip algorithms.
\newblock {\em IEEE/ACM Trans. Netw.}, 14(SI):2508--2530, June 2006.

\bibitem{bub15}
S.~Bubeck.
\newblock Convex optimization: Algorithms and complexity.
\newblock {\em Found. Trends Mach. Learn.}, 8(3-4):231--357, Nov. 2015.

\bibitem{Chang2011}
C.-C. Chang and C.-J. Lin.
\newblock Libsvm: a library for support vector machines.
\newblock {\em ACM transactions on intelligent systems and technology (TIST)},
  2(3):27, 2011.

\bibitem{dvinskikh2019decentralized}
D.~Dvinskikh and A.~Gasnikov.
\newblock Decentralized and parallelized primal and dual accelerated methods
  for stochastic convex programming problems.
\newblock {\em arXiv preprint arXiv:1904.09015}, 2019.

\bibitem{hendrikx2018accelerated}
H.~Hendrikx, F.~Bach, and L.~Massouli{\'e}.
\newblock Accelerated decentralized optimization with local updates for smooth
  and strongly convex objectives.
\newblock {\em arXiv preprint arXiv:1810.02660}, 2018.

\bibitem{lan2017communication}
G.~Lan, S.~Lee, and Y.~Zhou.
\newblock Communication-efficient algorithms for decentralized and stochastic
  optimization.
\newblock {\em Mathematical Programming}, pages 1--48, 2018.

\bibitem{lu2018geometrical}
Q.~L{\"u}, H.~Li, and D.~Xia.
\newblock Geometrical convergence rate for distributed optimization with
  time-varying directed graphs and uncoordinated step-sizes.
\newblock {\em Information Sciences}, 422:516--530, 2018.

\bibitem{Maros2018}
M.~Maros and J.~Jald{\'e}n.
\newblock Panda: A dual linearly converging method for distributed optimization
  over time-varying undirected graphs.
\newblock {\em 2018 IEEE Conference on Decision and Control (CDC)}, pages
  6520--6525, 2018.

\bibitem{nedic2014distributed}
A.~Nedi{\'c} and A.~Olshevsky.
\newblock Distributed optimization over time-varying directed graphs.
\newblock {\em IEEE Transactions on Automatic Control}, 60(3):601--615, 2014.

\bibitem{nedic2017geometrically}
A.~Nedi{\'c}, A.~Olshevsky, W.~Shi, and C.~A. Uribe.
\newblock Geometrically convergent distributed optimization with uncoordinated
  step-sizes.
\newblock In {\em 2017 American Control Conference (ACC)}, pages 3950--3955.
  IEEE, 2017.

\bibitem{Nedic2017achieving}
A.~Nedić, A.~Olshevsky, and W.~Shi.
\newblock Achieving geometric convergence for distributed optimization over
  time-varying graphs.
\newblock {\em SIAM Journal on Optimization}, 27(4):2597--2633, 2017.

\bibitem{Nemirovskii1983}
A.~Nemirovskii and Yudin.
\newblock {\em Problem complexity and method efficiency in optimization}.
\newblock Wiley, 1983.

\bibitem{nesterov2013introductory}
Y.~Nesterov.
\newblock {\em Introductory Lectures on Convex Optimization. A Basic Course.}
\newblock Springer Science \& Business Media, 2013.

\bibitem{Qu2017accelerated}
G.~{Qu} and N.~{Li}.
\newblock Accelerated distributed nesterov gradient descent.
\newblock {\em 2016 54th Annual Allerton Conference on Communication, Control,
  and Computing}, 2016.

\bibitem{rogozin2019optimal}
A.~Rogozin, C.~Uribe, A.~Gasnikov, N.~Malkovskii, and A.~Nedich.
\newblock Optimal distributed convex optimization on slowly time-varying
  graphs.
\newblock {\em IEEE Transactions on Control of Network Systems}, 2019.

\bibitem{bach2017optimal}
K.~Scaman, F.~Bach, S.~Bubeck, Y.~T. Lee, and L.~Massouli{\'e}.
\newblock Optimal algorithms for smooth and strongly convex distributed
  optimization in networks.
\newblock In D.~Precup and Y.~W. Teh, editors, {\em Proceedings of the 34th
  International Conference on Machine Learning}, volume~70 of {\em Proceedings
  of Machine Learning Research}, pages 3027--3036. PMLR, 2017.

\bibitem{shi15extra}
W.~Shi, Q.~Ling, G.~Wu, and W.~Yin.
\newblock Extra: An exact first-order algorithm for decentralized consensus
  optimization.
\newblock {\em SIAM Journal on Optimization}, 25(2):944--966, 2015.

\bibitem{shi2014on_the_linear}
W.~Shi, Q.~Ling, K.~Yuan, G.~Wu, and W.~Yin.
\newblock On the linear convergence of the admm in decentralized consensus
  optimization.
\newblock {\em IEEE Transactions on Signal Processing}, 62(7):1750--1761, 2014.

\bibitem{vandistributed}
B.~Van~Scoy and L.~Lessard.
\newblock A distributed optimization algorithm over time-varying graphs with
  efficient gradient evaluations.
\newblock {\em
  \url{https://vanscoy.github.io/docs/papers/vanscoy2019distributed.pdf}},
  2019.

\end{thebibliography}
\bibliographystyle{abbrv}

\newpage

\appendix

\section{Proof of Theorem \ref{th:decentralized_gd}}\label{ap:decentralized_gd}

\begin{lem}\label{lemma:extended_cauchy_schwarz}
	Let $u, v$ be vectors of $\RR^n$ of matrices of $\RR^{d\times n}$ and $p$ be a positive scalar constant. Then
	\begin{enumerate}
		\item 
		\begin{align}\label{eq:extended_cauchy_schwarz_1}
		\langle u, v\rangle \le \frac{\|u\|^2}{2p} + \frac{p\|v\|^2}{2}
		\end{align}
		\item If $p < 1$, then
		\begin{align}\label{eq:extended_cauchy_schwarz_2}
		\|v\|^2\ge p\|u\|^2 - \frac{p}{1-p}\|v - u\|^2
		\end{align}
	\end{enumerate}
	Here, if $u, v\in\RR^n$, $\|\cdot\|$ denotes the $2$-norm in $\RR^n$, and if $u, v\in\RR^{d\times n}$, $\|\cdot\|$ denotes Frobenius norm. 
\end{lem}
\begin{proof}
	\begin{enumerate}
		\item Multiplying both sides by $2c$ yields
		\begin{align*}
		\|u - pv\|^2 \ge 0.
		\end{align*}
		\item Analogously, multiplying both sides by $1-p$ leads to
		\begin{align*}
		(1-p)\|v\|^2 &\ge (p - p^2)\|u\|^2 \\
		&\qquad - p(\|v\|^2 + \|u\|^2 - 2\langle u, v\rangle) \\
		\|v\|^2 &\ge -p^2\|u\|^2 + 2p\langle u, v\rangle \\
		\|v - pu\|^2 &\ge 0
		\end{align*}
	\end{enumerate}
\end{proof}

\begin{lem}\label{lemma:proj_gd_step}
	Denote $\Pi_{k} = \proj_K(X_k), X^* = \Pi^* = \underset{K}{\argmin} F(X), r_k = \|\Pi_k - \Pi^*\| = \|\Pi_k - X^*\|$. Then for $r_k^2\ge \frac{12}{\hat\mu_f^2} L_{\max}^2 \eps_1$ it holds
	\begin{align*}
	r_{k+1}^2 \le r_k^2\left(1 - \frac{\hat\mu_f}{8\hat L_f}\right). 
	\end{align*}
\end{lem}

\begin{proof}
\begin{align*}
	&r_{k+1}^2 = \|\Pi_k - \Pi^* - \gamma\proj_K(\nabla F(X_k))\|^2 \\
	&\quad = r_k^2 + \gamma^2 \|\proj_K(\nabla F(X_k))\|^2 \\
	&\quad\quad - 2\gamma\langle \Pi_k - \Pi^*, \proj_K(\nabla F(X_k))\rangle \\
	&\quad = r_k^2 + \gamma^2 \|\proj_K(\nabla F(X_k))\|^2 \\
	&\quad\quad - 2\gamma\underbrace{\langle \Pi_k - \Pi^*, \proj_K(\nabla F(\Pi_k)) - \proj_K(\nabla F(X^*)) \rangle}_{\circledOne} \\
	&\quad\quad - 2\gamma\underbrace{\langle \Pi_k - \Pi^*, \proj_K(\nabla F(X_k)) - \proj_K(\nabla F(\Pi_k)) \rangle}_{\circledTwo} \numberthis\label{eq:proj_gd_step_bound_1}
\end{align*}

\begin{enumerate}
	\item First, let us estimate $\circledOne$. Note that for all $\Pi\in K$, it holds
	\begin{align*}
		\Pi &= [\pi, \ldots, \pi] \\
		\nabla F(\Pi) &= [\nabla f_1(\pi), \ldots, \nabla f_n(\pi)]
	\end{align*}
	Moreover, for all $X\in\RR^{d\times n}$ it holds
	\begin{align*}
		\proj_K(X) &= \argmin_{Z\in K} \|Z - X\|^2 = [\ol x, \ldots, \ol x],
	\end{align*}
	where $\ol x = \frac{1}{n}\Sum_{i=1}^n x_i$. In particular,
	\begin{align*}
		\proj_K(\nabla F(\Pi)) = [\nabla f(\pi) / n, \ldots, \nabla f(\pi) / n]
	\end{align*}
	
	Now we can estimate $\circledOne$ by Proposition \ref{prop:nesterov_strongly_convex_smooth}. For brevity we introduce $\hat\mu_f = \mu_f/n, \hat L_f = L_f/n$.
	\begin{align*}
		\langle &\Pi_k - \Pi^*, \proj_K(\nabla F(\Pi_k)) - \proj_K(\nabla F(X^*)) \rangle \\
		& = n\cdot \langle \pi_k - \pi^*, \nabla f(\pi) / n - \nabla f(x^*) / n\rangle \\
		& = \langle \pi_k - \pi^*, \nabla f(\pi) - \nabla f(\pi^*) \rangle \\
		&\ge \frac{(n\hat\mu_f)(n\hat L_f)}{n\hat\mu_f + n\hat L_f} \|\pi_k - \pi^*\|^2 \\
		&\qquad + \frac{1}{n\hat\mu_f + n\hat L_f} \|\nabla f(\pi_k) - \nabla f(\pi^*)\|^2 \\
		& = \frac{\hat\mu_f \hat L_f}{\hat\mu_f + \hat L_f} \|\Pi_k - \Pi^*\|^2 + \frac{1}{\hat\mu_f + \hat L_f} \|\proj_K(\nabla F(\Pi_k))\|^2
	\end{align*}
	
	\item Let us employ \eqref{eq:extended_cauchy_schwarz_1} with $p = \frac{\hat\mu_f + \hat L_f}{\hat\mu_f \hat L_f}$ to estimate $\circledTwo$.
	\begin{align*}
		&-2\gamma\langle \Pi_k - \Pi^*, \proj_K(\nabla F(X_k)) - \proj_K(\nabla F(\Pi_k)) \rangle \\
		&\qquad \le \gamma\cdot\frac{\hat\mu_f \hat L_f}{\hat\mu_f + \hat L_f} \|\Pi_k - \Pi^*\|^2 \\
		&\qquad\quad + \gamma\cdot\frac{\hat\mu_f + \hat L_f}{\hat\mu_f \hat L_f} \|\nabla F(X_k) - \nabla F(\Pi_k)\|^2 \\
		&\qquad \le \gamma\frac{\hat\mu_f \hat L_f}{\hat\mu_f + \hat L_f} r_k^2 + \gamma \frac{\hat\mu_f + \hat L_f}{\hat\mu_f \hat L_f} L_{\max}^2 \eps_1
	\end{align*}
\end{enumerate}

Now we return to \eqref{eq:proj_gd_step_bound_1}.
\begin{align*}
	r_{k+1}^2 
	&\le r_k^2 + \underline{\gamma^2\|\proj_K(\nabla F(X_k))\|^2} \\
	& - 2\gamma\frac{\hat\mu_f \hat L_f}{\hat\mu_f + \hat L_f} r_k^2 - \underline{\frac{2\gamma}{\hat\mu_f + \hat L_f} \|\proj_K(\nabla F(\Pi_k))\|^2} \\
	&\qquad + \gamma\frac{\hat\mu_f \hat L_f}{\hat\mu_f + \hat L_f} r_k^2 + \gamma\frac{\hat\mu_f + \hat L_f}{\hat\mu_f \hat L_f} L_{\max}^2 \eps_1 \numberthis \label{eq:proj_gd_step_bound_2}
\end{align*}

The sum of underlined terms may be estimated by setting $\gamma\in\left(0, \frac{2}{\hat\mu_f + \hat L_f}\right]$ and using \eqref{eq:extended_cauchy_schwarz_2} with $p = \gamma\frac{\hat\mu_f + \hat L_f}{2} \in (0, 1)$.
\begin{align*}
	\gamma^2 &\|\proj_K(\nabla F(X_k))\|^2 - \frac{2\gamma}{\hat\mu_f + \hat L_f} \|\proj_K(\nabla F(\Pi_k))\|^2 \\
	& = \frac{2\gamma}{\hat\mu_f + \hat L_f} \Bigg(\frac{\gamma(\hat\mu_f +\hat L_f)}{2} \|\proj_K(\nabla F(X_k))\|^2 \\
	&\qquad\qquad - \|\proj_K(\nabla F(\Pi_k))\|^2 \Bigg) \\
	&\le \frac{2\gamma}{\hat\mu_f + \hat L_f}\cdot \frac{\gamma(\hat\mu_f +\hat L_f)}{2}\cdot \left(1 - \frac{\gamma(\hat\mu_f + \hat L_f)}{2}\right)^{-1}\cdot \\
	&\qquad\|\proj_K(\nabla F(X_k)) - \proj_K(\nabla F(\Pi_k))\|^2 \\
	&\le \frac{2\gamma^2}{2 - \gamma(\hat\mu_f + \hat L_f)}\cdot L_{\max}^2 \eps_1
\end{align*}

Finally, we return to \eqref{eq:proj_gd_step_bound_2}, set $\gamma = \frac{1}{\hat\mu_f + \hat L_f}$ and estimate $r_{k+1}$.
\begin{align*}
	r_{k+1}^2
	&\le r_k^2 \left(1 - \gamma\frac{\hat\mu_f \hat L_f}{\hat\mu_f + \hat L_f}\right) \\
	&\qquad + \left(\gamma\frac{\hat\mu_f + \hat L_f}{\hat\mu_f \hat L_f} + \frac{2\gamma^2}{2 - \gamma(\hat\mu_f + \hat L_f)}\right)\cdot L_{\max}^2 \eps_1 \\
	&= r_k^2 \left(1 - \frac{\hat\mu_f \hat L_f}{(\hat\mu_f + \hat L_f)^2}\right) + \left(\frac{1}{\hat\mu_f \hat L_f} + \frac{2}{(\hat\mu_f + \hat L_f)^2}\right)\cdot L_{\max}^2 \eps_1 \\
	&\le r_k^2 \left(1 - \frac{\hat\mu_f}{4\hat L_f}\right) + \frac{3}{2\hat\mu_f \hat L_f}\cdot L_{\max}^2 \eps_1.
\end{align*}
For $r_k^2\ge \frac{12}{\hat\mu_f^2} L_{\max}^2 \eps_1$ it holds
\begin{align*}
	r_{k+1}^2 \le r_k^2\left(1 - \frac{\hat\mu_f}{8\hat L_f}\right). 
\end{align*}
\end{proof}

\begin{proof}[Proof of Theorem \ref{th:decentralized_gd}]
Due to Lemma \ref{lemma:proj_gd_step}, after $N$ iterations we get 
\begin{align*}
	\|\Pi_N - X^*\| = \|\Pi_N - \Pi^*\|^2 = r_N^2
	\le \frac{12}{\mu_f^2}L_{\max}^2\eps_1.
\end{align*}
Since $\|X_N - \Pi_N\|^2\le\eps_1$,
\begin{align*}
	\|X_N - X^*\|^2 &= \|X_N - \Pi_N\|^2 + \|\Pi_N - X^*\|^2 \\
	&\le \frac{12}{\mu_f^2}L_{\max}^2\eps_1 + \eps_ 1 \le \frac{13}{\mu_f^2}L_{\max}^2\eps_1 = \eps,
\end{align*}
which concludes the proof.
\end{proof}

\section{Proof of Theorem \ref{th:summarize_rate}}

\begin{lem}\label{lem:steps_to_consensus}
	Let $r_0 = \norm{X_0 - \proj_K (X_0)}$ and let Assumption \ref{assum:mixing_matrix} hold. Then
	\begin{enumerate}
		\item At every step $k$ of Algorithm \ref{alg:consensus}, it holds $\proj_K X_k = \proj_K X_0$.
		\item For any $\eps> 0$, after $m \ge B\left(\log(1/\delta)\right)^{-1} \log(r_0/\eps)$ iterations Algorithm \ref{alg:consensus} yields $X_m$ such that $\norm{X_m - \proj_K (X_m)} \le \eps$.
	\end{enumerate}
\end{lem}
\begin{proof}
	\item
	\begin{enumerate}
		\item Introduce projection matrix $P = \frac{1}{n} \onevec\onevec^\top$ and note that $XP = \proj_K X$ for every $X\in\RR^{d\times n}$. Due to Assumption \ref{assum:mixing_matrix}, $W(k)$ is doubly stochastic for every $k$ and therefore $PW(k) = P$. This means that $X_k P = X_{k-1} P = \ldots = X_{0} P = \proj_K X_0$.
		\item For $m \ge B\left(\log(1/\delta)\right)^{-1} \log(r_0/\eps)$, it holds
		\begin{align*}
			&\frac{m}{B}\log(1/\delta)\ge \log(r_0/\eps), \\
			&\left(\frac{1}{\delta}\right)^{m/B}\ge \frac{r_0}{\eps}, \\
			&r_0\cdot \delta^{m/B}\le \eps.
		\end{align*}
		By Proposition \ref{prop:consensus} it holds that $\norm{X_m - \proj_K X_m}\le \delta^{m/B}\norm{X_0 - \proj_K X_0}\le\eps$.
		\begin{align*}			
		\end{align*}
	\end{enumerate}
\end{proof}

\begin{proof}[Proof of Theorem \ref{th:summarize_rate}]
It is sufficient to show that on every outer iteration, the consensus algorithm requires $O\left(B\left(\log\left(\frac{1}{\delta}\right)\right)^{-1} \log\left(\frac{\left(\|\nabla F(X^*)\| + L_{\max}\norm{X_0 - X^*}\right) n L_{\max}^2}{\eps\mu_f^2 L_f}\right)\right)$ iterations. In order to do this, let us estimate the distance to consensus after taking a gradient step at iteration $k$.

Denote $X_{k+1/2} = X_k - \gamma\nabla F(X_k)$, where $\gamma = \frac{n}{L_f + \mu_f}$. Then
\begin{align*}
	&\norm{X_{k+1/2} - \proj_K(X_{k+1/2})} = 
	\norm{X_k - \gamma\nabla F(X_k) - \proj_K(X_{k} - \gamma\nabla F(X_k))} \\
	&\qquad\le \norm{X_k - \proj_K(X_k)} + \gamma\norm{\nabla F(X_k) - \proj_K(\nabla F(X_k))} \le \eps_1 + \gamma\norm{\nabla F(X_k)}.
\end{align*}
We estimate $\nabla F(X_k)$ using $L_{\max}$-smoothness of $F$ on $\RR^d$ and Lemma \ref{lemma:proj_gd_step}.
\begin{align*}
	&\norm{\nabla F(X_k)}\le \norm{\nabla F(X^*)} + \norm{\nabla F(X_k) - \nabla F(X^*)} \le \norm{\nabla F(X^*)} + L_{\max}\norm{X_k - X^*} \\
	&\quad\le \norm{\nabla F(X^*)} + L_{\max}(\norm{X_k - \Pi_k} + r_k) = \norm{\nabla F(X^*)} + L_{\max}\eps_1 + L_{\max}r_0.
\end{align*}
Now the distance to consensus at iteration $k$ is estimated as
\begin{align*}
	&\norm{X_{k+1/2} - \proj_K(X_{k+1/2})}\le \eps_1(1 + \gamma L_{\max}) + \gamma\norm{\nabla F(X^*)} + \gamma L_{\max}r_0 \\
	&= \eps_1(1 + \gamma L_{\max}) + (\norm{\nabla F(X^*)} + L_{\max}r_0)\cdot\frac{n}{L_f + \mu_f}.
\end{align*}
By Lemma \ref{lem:steps_to_consensus}, the number of iterations required by the consensus algorithm is
\begin{align*}
	N 
	&= O\left(B\left(\log\left(1/\delta\right)\right)^{-1}\frac{\norm{X_{k+1/2} - \proj_K(X_{k+1/2})}}{\eps_1}\right) \\
	&= O\left(B\left(\log\left(1/\delta\right)\right)^{-1}  \log\left(\frac{\left(\|\nabla F(X^*)\| + L_{\max}\norm{X_0 - X^*}\right) n L_{\max}^2}{\eps\mu_f^2}\right)\right)
\end{align*}
which concludes the proof.
\end{proof}

\end{document}